\documentclass[11pt]{amsart}

\usepackage{latexsym}
\usepackage{amsmath}
\usepackage{amsfonts}
\usepackage{amssymb}
\usepackage{enumitem}
\usepackage{blkarray}
\usepackage {tikz}
\usepackage{color}
\usepackage[noend]{algorithmic} 
\usepackage{algorithm,caption}
\algsetup{indent=2em}

\makeatletter
\newdimen\p@renwd \setbox0=\hbox{\phantom B} \p@renwd=\wd0
\def\bordersquare#1{\begingroup \m@th
  \setbox0=\vbox{\def\cr{\crcr\noalign{\kern2pt\global\let\cr=\endline}}
      \ialign{$##$\hfil\kern2pt\kern\p@renwd&\thinspace\hfil$##$\hfil
        &&\quad\hfil$##$\hfil\crcr
        \omit\strut\hfil\crcr\noalign{\kern-\baselineskip}
        #1\crcr\omit\strut\cr}}
  \setbox2=\vbox{\unvcopy0 \global\setbox1=\lastbox}
  \setbox2=\hbox{\unhbox1 \unskip \global\setbox1=\lastbox}
  \setbox2=\hbox{$\kern\wd1\kern-\p@renwd \left[ \kern-\wd1
    \global\setbox1=\vbox{\box1\kern2pt}
    \vcenter{\kern-\ht1 \unvbox0 \kern-\baselineskip} \,\right]$}
  \null\;\vbox{\kern\ht1\box2}\endgroup}
\makeatother

\usetikzlibrary {positioning}
\usetikzlibrary{arrows,shapes}
\usepackage{color}
\newtheorem{theorem}{Theorem}[section]
\newtheorem{lemma}[theorem]{Lemma}
\newtheorem{corollary}[theorem]{Corollary}
\newtheorem{proposition}[theorem]{Proposition}

\newtheorem{sublemma}{}[theorem]

\theoremstyle{definition}

\theoremstyle{remark}

\numberwithin{equation}{section}
\usepackage{stackengine}

\stackMath

\newcommand{\ba}{\backslash}

\usepackage{mathtools}

\begin{document}

\title[Comatroids]{The smallest class of binary matroids closed under direct sums and complements}

\author{James Oxley}
\address{Mathematics Department\\
Louisiana State University\\
Baton Rouge, Louisiana}
\email{oxley@math.lsu.edu}

\author{Jagdeep Singh}
\address{Mathematics Department\\
Louisiana State University\\
Baton Rouge, Louisiana}
\email{jsing29@math.lsu.edu}

\date{\today}

\begin{abstract}

The class of cographs or complement-reducible graphs is the class of graphs that can be generated from $K_1$ using the operations of disjoint union and complementation. By analogy, this paper introduces the class of binary comatroids as the class of matroids that can be generated from the empty matroid using the operations of direct sum and taking complements inside of binary projective space. We show that a proper flat of a binary comatroid is a binary comatroid. Our main result identifies those binary non-comatroids for which every proper flat is a binary comatroid. The paper also proves the corresponding results for ternary matroids.

\end{abstract}

\maketitle

\section{Introduction}

The notation and terminology used in this paper follow
\cite{diestel} and \cite{ox1} except where otherwise indicated. All graphs and matroids considered here are \textbf{simple}. A \textbf{cograph} is defined recursively as follows:

\begin{enumerate}[label=(\roman*)]
    \item $K_1$ is a cograph;
    \item if $G_1$ and $G_2$ are cographs, then so is their disjoint union; and
    \item if $G$ is a cograph, then so is its complement.
\end{enumerate}

The class of cographs has been extensively studied over the last fifty years (see, for example, \cite{ corneil2, jung, sein}). In particular, there are numerous equivalent characterizations of cographs including that a cograph is a graph in which every connected induced subgraph has a disconnected complement.

Our goal in this paper is to give a matroid analogue of cographs by, loosely speaking, considering the smallest class of matroids that is closed under direct sums and complementation. One immediate obstacle to achieving this goal is that  matroids in general do not have complements. However, if $M$ is a simple uniquely $GF(q)$-representable matroid and $k \geq r(M)$, the $(GF(q),k)$-\textbf{complement} of $M$ is the matroid $PG(k-1,q) \ba T$ where $M \cong PG(k-1,q)|T$. Brylawski and Lucas \cite{brwlucas} (see also \cite[Proposition~10.1.7]{ox2}) showed that this $(GF(q),k)$-complement of $M$ is well-defined. By convention, we write $M^c$ for the $(GF(q), r(M))$-complement of $M$. Although a $GF(q)$-representable matroid need not be uniquely representable when $q \geq 4$, it is uniquely representable for $q$ in $\{2,3\}$. Thus we only introduce analogues of cographs for binary and ternary matroids. In particular, for $q$ in $\{2,3\}$, we define a $GF(q)$-\textbf{comatroid} recursively as follows:

\begin{enumerate}[label=(\roman*)]
    \item $U_{0,0}$ is a $GF(q)$-comatroid;
    \item if $M_1$ and $M_2$ are $GF(q)$-comatroids, then so is their direct sum; and
    \item if $M$ is a $GF(q)$-comatroid, then so is its $(GF(q),t)$-complement for all $t \geq r(M)$.
\end{enumerate}

As $PG(r-1,q)$ is the $(GF(q),r)$-complement of $U_{0,0}$, every projective geometry $PG(r-1,q)$ for $r \geq -1$ is a $GF(q)$-comatroid. In particular, as $U_{1,1}$ is $PG(0,q)$, we see that $U_{n,n}$ is a $GF(q)$-comatroid for all $n \geq 0$. We sometimes call $GF(2)$- and $GF(3)$-comatroids, \textbf{binary} and \textbf{ternary comatroids}, respectively.

The following characterization of $GF(q)$-comatroids is particularly useful.

\begin{theorem}
\label{equiv}
For $q$ in $\{1,2\}$, a simple $GF(q)$-representable matroid $M$ is a $GF(q)$-coma\-troid if and only if $M$ is $U_{0,0}$ or, for all flats $F$ of $PG(r(M)-1,q)$ with $r(F \cap E(M)) = r(F-E(M))$, the restriction of $PG(r(M)-1,q)$ to either $F\cap E(M)$ or $F-E(M)$ is disconnected.
\end{theorem}

Corneil, Lerchs, and Stewart \cite{corneil} proved that a graph $G$ is a cograph if and only if $G$ does not have the $4$-vertex path as an induced subgraph. The next two theorems, which are our main results, prove matroid analogues of this theorem for binary and ternary comatroids by using the fact that a set $X$ of edges in a graph $H$ is the edge set of an induced subgraph of $H$ if and only if $X$ is a flat of $M(H)$.  The matroid 
$P(U_{3,4},U_{3,4})$, the parallel connection of two $4$-circuits, is the cycle matroid of a $6$-cycle with a single chord where this chord lies in two $4$-cycles.


\begin{theorem}
\label{main_binary}
A binary matroid $M$ is a   binary  comatroid if and only if neither $M$ nor $M^c$ has a  flat isomorphic to   a circuit of size exceeding five, to $P(U_{3,4},U_{3,4})$, or to the cycle matroid   of one of the six $5$-vertex graphs shown in Figure~\ref{main_M_2}.
\end{theorem}

\begin{figure}[h]
\center
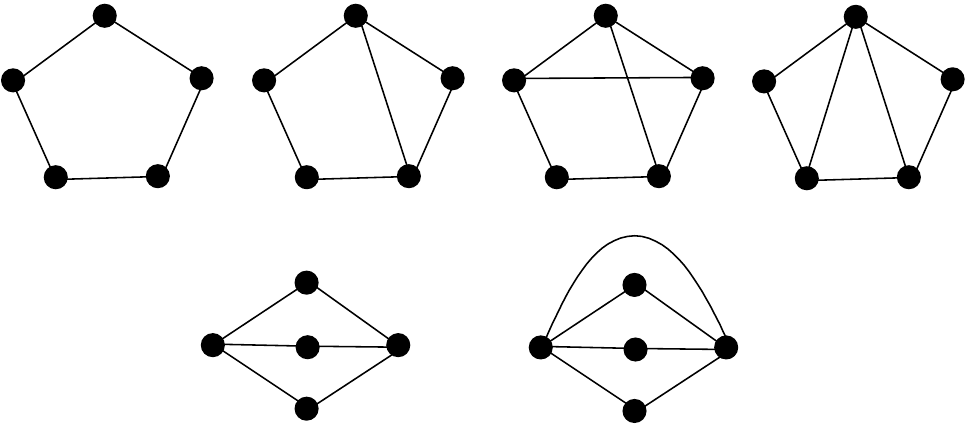
\caption{The cycle matroids of these graphs are induced-restriction-minimal binary non-comatroids.}
\label{main_M_2}
\end{figure}

\begin{theorem}
\label{main_ternary}
A ternary matroid $M$ is a   ternary  comatroid if and only if neither $M$ nor $M^c$ has a  flat isomorphic to a circuit of size exceeding three, to  a matroid that can be obtained from a circuit of size at least three by $2$-summing a copy of $U_{2,4}$ to at least one of the elements of the circuit, or to one of the  five rank-$3$ matroids   $P(U_{2,3},U_{2,3}),  U_{2,4} \oplus_2 U_{2,4}, P(U_{2,4},U_{2,3}),$  $M(K_4),$ and  $\mathcal{W}^3$.
\end{theorem}

The proofs of these theorems are given in Sections~\ref{bin} and \ref{tern}, respectively. In Section~\ref{prelim}, we prove a number of preliminary results including Theorem~\ref{equiv}. In particular, we show that if we contract an element from a $GF(q)$-comatroid and simplify the resulting matroid, then we get another $GF(q)$-comatroid. 
A simple matroid $N$ is an \textbf{induced minor} of a simple matroid $M$ if  $N$ can be obtained from $M$ by a sequence of operations each of which consists of restricting to a flat, or contracting an element and then simplifying. As consequences of Theorems~\ref{main_binary} and \ref{main_ternary}, we have the following characterizations of binary and ternary comatroids in terms of forbidden induced minors.

\begin{corollary}
\label{cor_binary}
A binary matroid   is a  binary comatroid if and only if   it  has no    induced minor  isomorphic to   the complement of a circuit of size exceeding five, to $P(U_{3,4},U_{3,4})$, to  the cycle matroid  of one of the six $5$-vertex graphs  in Figure~\ref{main_M_2}, or to the  complements of these cycle matroids in $PG(3,2)$.
\end{corollary}

\begin{corollary}
\label{cor_ternary}
A ternary matroid   is a  ternary comatroid if and only if  it  has no induced minor isomorphic to any of the following: the complements of  all matroids that can be obtained from a circuit of size at least three by $2$-summing a copy of $U_{2,4}$ to some, possibly empty, set of elements of the circuit;  the    matroids,   $U_{3,4},$ $P(U_{2,3},U_{2,3}),$ $U_{2,4} \oplus_2 U_{2,3},$ $U_{2,4} \oplus_2 U_{2,4},$  $P(U_{2,4},U_{2,3}),$  $M(K_4),$ and $\mathcal{W}^3 $; 
or the complements of these  matroids in $PG(2,3)$.
\end{corollary}

\section{Preliminaries}
\label{prelim}

 Throughout the paper, we call cocircuits, flats, and hyperplanes of $PG(r-1,q)$ \textbf{projective cocircuits}, \textbf{projective flats},  and \textbf{projective hyperplanes}, respectively.
 Viewing a $GF(q)$-representable matroid $M$ as a restriction of $PG(r(M)-1, q)$, we color the elements of $E(M)$ green while assigning the color red to the elements of $PG(r(M)-1,q)$ not in $E(M)$. We will frequently use $G$ and $R$ to denote both the sets of green and red elements  and the matroids obtained by restricting $PG(r(M)-1,q)$  to these sets of elements. The next lemma is an immediate consequence of the fact that the elements of a projective geometry are not all contained in two hyperplanes.
 
 \begin{lemma}
 \label{cover}
 For an arbitrary prime power $q$, let $(G,R)$ be a $2$-coloring of $PG(r-1,q)$. Then $r(G) = r$ or $r(R) = r$.
 \end{lemma}

\begin{proposition}
\label{sum_of_connectivities}
For an arbitrary prime power $q$, let $(G,R)$ be a $2$-coloring of $PG(r-1,q)$. Let $j$ and $k$ be the vertical connectivities of $G$ and $R$, respectively. Then 
$j+k \geq r$ unless $(q,r) = (2,3)$ and $\{G,R\} = \{U_{3,3}, U_{2,3} \oplus U_{1,1}\}$.
\end{proposition}

\begin{proof}
Assume that the result fails. If $R$ is empty, then $k=0$ and $j=r$, a contradiction. Thus we may assume that $G$ and $R$ are both non-empty. Then $j$ and $k$ are both non-zero, so we may assume that $j,k \in \{1,2,\dots, r-2\}$. Observe that if $r(R) < r$, then $G$ contains $AG(r-1,q)$ and hence $j \geq r-1$, a contradiction. Therefore $r(G)=r(R)=r$. Thus $G$ has  an exact vertical $j$-separation $(A,B)$   with $r(A) \geq r(B)$, and $R$ has  an exact vertical $k$-separation $(X,Y)$ with $r(X) \geq r(Y)$.  Let $r(A) = a$. Then
\begin{equation}
\label{g1}
|G| \leq \frac{q^{a}-1}{q-1} + \frac{q^{a}-1}{q-1} - \frac{q^{j-1}-1}{q-1} = \frac{2q^{a}- q^{j-1} -1}{q-1}. 
\end{equation}
By symmetry, with $r(X) = x$, we have 
\begin{equation}
\label{r1}
|R| \leq  \frac{2q^{x}- q^{k-1} -1}{q-1}.
\end{equation}

First suppose that $x = r-1$. Then $r(Y) = k$. Let $H_X$ be the projective hyperplane spanned by $X$. Observe that the intersection of $H_X$ with the projective flat $F_Y$ spanned by $Y$  is a projective flat of rank $k-1$. Thus, as $R-H_X \subseteq F_ Y - (F_Y \cap H_X)$, we have

\begin{equation}
\label{bound_R_first}
|R - H_X| \leq \frac{q^{k}-1}{q-1} - \frac{q^{k-1}-1}{q-1} = q^{k-1}. 
\end{equation}

Suppose that $a=r-1$. Then $r(B) = j$ and so, as for (\ref{bound_R_first}), we have  
$|G - H_A| \leq 
q^{j-1}$ where $H_A$ is the projective hyperplane spanned by $A$. Note that $E(PG(r-1,q))-(H_A \cup H_X)$ has at least $q^{r-2}$ elements and so it follows by (\ref{bound_R_first}) that $q^{k-1} + q^{j-1} \geq q^{r-2}$. Since $j$ and $k$ are in 
$\{1,2,\dots,r-2\}$, this is a contradiction unless $(q,r) = (2,3)$ and $k= 1 = j$ and $r= 3$. In the exceptional case, it is straightforward to check that $\{G,R\} = \{U_{3,3}, U_{2,3} \oplus U_{1,1}\}$, and we get the exceptional case noted in the proposition.

We may now assume that $a < r-1$. Let $F_A$ be the projective flat spanned by $A$. Observe that $F_A \cap H_X$ is a projective flat of rank $a$ or $a-1$. Thus 
\begin{eqnarray*}
|(R\cup G)-H_X| & = &|R- H_X| + |G - H_X|\\
& \leq &|R- H_X| + |G| - |F_A \cap H_X|\\
& \leq  &q^{k-1} + \frac{2q^{a} - q^{j-1}-1}{q-1} - \frac{q^{a-1}-1}{q-1}
\end{eqnarray*} 
where the last step follows by (\ref{bound_R_first}) and (\ref{g1}). 
As  $|(R\cup G)-H_X| = |E(AG(r-1,q)| = q^{r-1}$, we have $q^{k-1} + \frac{2q^{a} - q^{j-1}-1}{q-1} - \frac{q^{a-1}-1}{q-1} \geq q^{r-1}$. As the left-hand side of the last inequality is bounded above by  $q^{r-3} + \frac{q^{r-3}(2q-1)}{q-1} - \frac{q^{j-1}-1}{q-1}$, we deduce that $1 + \frac{2q-1}{q-1} > q^2$. This is a contradiction as $q \ge 2$.  We conclude that  $x < r-1$. By symmetry, 
 $a < r-1$. Then, by (\ref{r1}) and  (\ref{g1}), $|R| + |G| < \frac{q^r-1}{q-1}$, a contradiction. 
\end{proof}

Next, we move towards proving Theorem~\ref{equiv}. We omit the straightforward proof of the following result. 

\begin{lemma}
\label{comat_small_rank}
Let $M$ be a simple $GF(q)$-representable matroid.

\begin{enumerate}[label=(\roman*)]
    \item If $q=2$ and $r(M) \leq 3$, then $M$ is a $GF(q)$-comatroid.
    \item If $q=3$ and $r(M) \leq 2$, then $M$ is a $GF(q)$-comatroid.
\end{enumerate}
\end{lemma}



\begin{lemma}
\label{closure_under_components}
Let $M$ be a $GF(q)$-comatroid and suppose that $M$ is disconnected. Then each of its components is a $GF(q)$-comatroid.
\end{lemma}

\begin{proof}
By Lemma~\ref{comat_small_rank}, we may assume that $r(M) \geq 4$. Let $M=M_1 \oplus M_2$. 
Take a shortest sequence of direct sums and complements that shows that $M$ is a $GF(q)$-comatroid. 
Assume that the final step in creating $M$ is not a direct sum. Then this final step involves taking the $(GF(q),t)$-complement of some matroid $N_1$ 
where $t \ge r(N_1)$. As $M$ is disconnected, $t  = r(N_1) \ge r(M)$, otherwise $M$ has $AG(t-1,q)$ as a restriction and so is connected.
Thus $N_1^c = M$. 
Moreover, as the vertical connectivity of $M$ is one, Proposition \ref{sum_of_connectivities} implies that $N_1$ is connected. Since $N^c_1 = M$, the predecessor of $N_1$ in the construction of $M$ is its $(GF(q),s)$-complement $N_2$ for some $s \geq r(N_1)+1$. Then $N_2$ has $AG(s-1,q)$ as a restriction, so it is connected. The predecessor of $N_2$ in the production of $M$ must again be a connected matroid $N_3$ of rank exceeding $r(N_2)$. Tracing back the predecessors of $M$ in its creation as a $GF(q)$-comatroid, we obtain an infinite sequence of matroids of increasing ranks. This contradicts the fact that $M$ is created by a finite process. We conclude that, when $M$ is a disconnected  $GF(q)$-comatroid, the final step in constructing it is taking the direct sum of two $GF(q)$-comatroids. Thus if $M$ has exactly two components, then each component is a $GF(q)$-comatroid. We now argue by induction on the number of components of $M$. As the final step in creating $M$ is taking a direct sum of two $GF(q)$-comatroids, it follows by induction that each component of $M$ is a $GF(q)$-comatroid. 
\end{proof}

\begin{lemma}
\label{not_comatroid_1}
For $q$ in $\{2,3\}$, a $GF(q)$-representable matroid $M$ such that $r(M)=r(M^c)$ and both $M$ and $M^c$ are connected is not a $GF(q)$-comatroid.
\end{lemma}

\begin{proof}
For $M$ to satisfy the hypotheses of the lemma, we must have $r(M) \geq 3$. Moreover, $r(M) \geq 4$ if $q=2$. Assume that $M$ is a $GF(q)$-comatroid. 
Again  we take a shortest sequence of direct sums and complements that shows that $M$ is a $GF(q)$-comatroid. 
Then the final step in creating $M$ must have been taking a complement. As $r(M)=r(M^c)$, for some $N_0$ in $\{M,M^c\}$, the predecessor of $N_0$ in the creation of $M$ is the $(GF(q),t)$-complement $N_1$ of $N_0$ for some $t > r(M)$. This matroid is also connected. Its predecessor in the construction of $M$ is the $(GF(q),s)$-complement $N_2$ of $N_1$ for some $s > r(N_1)$. Again, $N_2$ is connected and this process must continue indefinitely, a contradiction. 
\end{proof}

As an immediate consequence of the last lemma, we have the following.

\begin{corollary}
\label{cct}
A $k$-circuit is   a $GF(q)$-comatroid if and only if $q+k \le 6$.
\end{corollary}

\begin{lemma}
\label{closure_under_flats}
The restriction of a $GF(q)$-comatroid to one of its flats is a $GF(q)$-comatroid. 
\end{lemma}

\begin{proof}
We argue by induction on the rank of the $GF(q)$-comatroid $M$. The result holds by Lemma \ref{comat_small_rank} if $r(M) \leq 2$. 
 Now assume it holds for every $GF(q)$-comatroid of rank less than $n$ and let $M$ be a $GF(q)$-comatroid of rank $n$ where $n \ge 3$. Then $M$ is obtained by taking complements and direct sums. Let $F$ be a proper flat of $M$. Assume first that $M$ is disconnected. Then, by Lemma~\ref{closure_under_components} and the induction assumption, $M|(F \cap E(M_i))$ is a $GF(q)$-comatroid for each component $M_i$ of $M$. Thus  $M|F$ is a $GF(q)$-comatroid. We may now assume that $M$ is connected. Suppose $N=M^c$. Then $N$ is a $GF(q)$-comatroid.  Let $F_P$ be the projective flat of $PG(r(M)-1,q)$ that is spanned by $F$. Then $F_P-F$ is a flat of $N$. The complement of $N|(F_P-F)$ in $F_P$ is $M|F$. Assume first that $r(N)<r(M)$. Then, by the induction assumption, $N|(F_P-F)$ is a $GF(q)$-comatroid. Thus $M|F$ is also a $GF(q)$-comatroid. Hence we may assume that $r(N)=r(M)$. By Lemma~\ref{not_comatroid_1}, $N$ is not connected, so, by the induction assumption, $N|(F_P-F)$ and hence $M|F$ is a $GF(q)$-comatroid.
\end{proof}

\begin{proof}[Proof of Theorem \ref{equiv}]
Suppose $M$ is a non-empty $GF(q)$-comatroid. By Lemma \ref{closure_under_flats}, if $F$ is a flat of $PG(r(M)-1,q)$, then $M|(F\cap E(M))$ is a $GF(q)$-comatroid. Hence so is  its complement in $F$, namely $PG(r(M)-1,q)|(F-E(M))$. 
By Lemma~\ref{cover}, at least one of $r(F\cap E(M))$ and $r(F-E(M))$ is $r(F)$. 
Thus, by Lemma \ref{not_comatroid_1}, if $r(F\cap E(M))=r(F-E(M))$, then the restriction of $PG(r(M)-1,q)$ to $F \cap E(M)$ or $F-E(M)$ is disconnected.

Conversely, suppose that $M$ is non-empty and that, for all flats $F$ of $PG(r(M)-1,q)$ with $r(F \cap E(M)) = r(F - E(M))$, the restriction of $PG(r(M)-1,q)$ to either $F \cap E(M)$ or $F - E(M)$ is disconnected. We argue by induction on $r(M)$ that $M$ is a $GF(q)$-comatroid. By Lemma~ \ref{comat_small_rank}, this is true if $r(M) \leq 2$. Assume it is true for $r(M) < n$ and let $r(M)=n$. If $M$ is disconnected, then, by the induction assumption,  each component is a $GF(q)$-comatroid. Hence so too is $M$. Thus $M$ is connected. Suppose $r(M) = r(M^c)$. Then, by the hypothesis, $M^c$ is disconnected. Since $M^c$ obeys the same condition as $M$, each of its components is a $GF(q)$-comatroid. Thus so is $M^c$. Hence $M$ is a $GF(q)$-comatroid. We may now assume that $r(M) > r(M^c)$. Take $F_0$ to be the flat of $PG(r(M)-1,q)$ spanned by $M^c$. Let $F_1$ be a flat of $F_0$ with $r(F_1 \cap E(M^c)) = r(F_1 - E(M^c))$. Then $r(F_1 \cap E(M)) = r(F_1 - E(M))$ so the restriction of $PG(r(M)-1,q)$ to $F_1 \cap E(M)$ or $F_1 - E(M)$ is disconnected. Thus the restriction of $F_0$ to $F_1 \cap E(M^c)$ or $F_1 - E(M^c)$ is disconnected. We conclude that $M^c$ obeys the same condition as $M$, so, by the induction assumption, $M^c$ is a $GF(q)$-comatroid. The $(GF(q),r(M))$-complement of $M^c$ is $M$ so it too is a $GF(q)$-comatroid, as required.
\end{proof}

In the following result, we note that if a $GF(q)$-comatroid is connected, it is highly connected.

\begin{proposition}
\label{comat_conn}
Let $M$ be a connected $GF(q)$-comatroid of rank $r$. Then $M$ is vertically $(r-1)$-connected.
\end{proposition}

\begin{proof}
By Lemma \ref{not_comatroid_1}, $M^c$ is either disconnected or has rank less than $r$. If $M^c$ is disconnected, then, by Proposition \ref{sum_of_connectivities}, $M$ is vertically $(r-1)$-connected. We may now assume that $M^c$ has rank less than $r$. Then $M$ is an extension of $AG(r-1,q)$. Since this affine geometry is vertically $(r-1)$-connected, the result follows.
\end{proof}

Next we show that the class of $GF(q)$-comatroids is closed under induced minors. For a subset $X$ of the ground set of a simple $GF(q)$-representable matroid $M$, we say $X$ is \textbf{connected} if $M|X$ is connected. When $X$ is a flat of $M$, we denote by $X^c$ the matroid $(M|X)^c$  that is obtained from the projective closure of $X$ by deleting $X$.

\begin{proposition}
\label{comat_contract}
Every induced minor of a $GF(q)$-comatroid is a $GF(q)$-comatroid.
\end{proposition}

\begin{proof}
By Lemma \ref{closure_under_flats}, the restriction of a $GF(q)$-comatroid to one of its flats is a $GF(q)$-comatroid. Now take an element $e$ of $M$ and assume that ${\rm si}(M/e)$ is not a $GF(q)$-comatroid.
View ${\rm si}(M/e)$ as a restriction of $PG(r(M)-2,q)$. Then, by Theorem \ref{equiv}, there is a flat $F$ of ${\rm si}(M/e)$ such that $F$ and $F^c$ are both connected and each has rank $k$, say. Observe that ${\rm cl}_M(F \cup e)$ is a connected flat of $M$ of rank $k+1$ unless $e$ is a coloop of $M|(F \cup e)$. In the exceptional case, $F$ is a flat of $M$ and, therefore, $M$ has a flat $F$ such that both $F$ and $F^c$ are connected of rank $k$, a contradiction. We deduce that ${\rm cl}_M(F \cup e)$ is a connected flat of rank $k+1$. We complete the proof by establishing the contradication that  the complement of ${\rm cl}_M(F \cup e)$ is also connected of rank $k+1$. To see this, note that, for each element $g$ of $F^c$, all the elements apart from $e$ that are on the projective line containing $\{e,g\}$ in $PG(r(M) - 1,q)$ are in the complement of  ${\rm cl}_M(F \cup e)$. Thus this complement contains a set  of rank $k+1$ that is a union of interlocking $4$-circuits. Hence this complement is connected.
\end{proof}

\section{Connected hyperplanes}
\label{ch}

Kelmans~\cite{kel} and Seymour (in \cite{ox78})   independently established that if $M$ is a simple connected binary matroid that   has no cocircuits of size less than three, then $M$ has a connected hyperplane. That theorem was extended in several ways by McNulty and Wu~\cite{mwu}. In this section, we note two of these extensions and prove an analogue for ternary matroids of the result of Kelmans and Seymour. These results on connected hyperplanes will be crucial in proving our characterizations of binary and ternary comatroids.

We begin the section by identifying when there is a free element in a binary or ternary matroid, where an element $e$ is {\it free} in a matroid $M$ if $e$ is not a coloop of $M$ and the only circuits that contain $e$ are spanning. 
Doubtless, the results in the next lemma are known but we include the proofs for completeness. In a rank-zero matroid, every element is free. In a rank-one matroid, the  free elements are the non-loops unless the matroid has a coloop in which case there are no free elements. Thus the next result only considers  matroids $M$ of rank at least two noting that $e$ is free in $M$ if and only if $e$ is free in ${\rm si}(M)$  and $e$ is in no $2$-circuits of $M$.

\begin{lemma}
\label{freedom} 
Let $M$ be a simple $GF(q)$-representable matroid of rank at least two and let $e$ be a free element of $M$. Then 
\begin{itemize}
\item[(i)] $M$ is a circuit when $q =2$; and
\item[(ii)] when $q = 3$, either $M \cong U_{2,4}$, or $M$ can be obtained from a circuit $C$ containing $e$ by, for some subset $D$ of $C- e$, $2$-summing a copy of $U_{2,4}$ across each element of $D$.
\end{itemize}
\end{lemma}

\begin{proof}
Since $e$ is free in $M$,    there is a spanning circuit $C_0$ of $M$ containing $e$. Then $M|C_0$ is represented over $GF(q)$ by $[I_r| {\bf 1}]$ where ${\bf 1}$,  the column of all ones,  is labelled by $e$. When $q= 2$, we cannot add any further elements without creating either a $2$-circuit, or a circuit that contains $e$ and has fewer than $r+ 1$ elements. Thus (i) holds. 

Now suppose that $q = 3$. If $r(M) = 2$, then $M$ is isomorphic to $U_{2,3}$ or $U_{2,4}$. Assume that $r(M) \ge 3$. Let $Z$ be a matrix representing $M$ over $GF(3)$ and having $[I_r| {\bf 1}]$ as its first $r+1$ columns. We will write the elements of $GF(3)$ as $0,1$, and $-1$. Let $f$ be a column of $Z$ other than one of the first $r+1$ columns.  As $M$ is simple, $f$ has at least two non-zero entries. If $f$ has two  non-zero entries with a common sign, then there is a circuit containing $\{e,f\}$ having at most $r$ elements, a contradiction. It follows that $f$ has exactly two non-zero entries and that these entries have different signs. If columns $f$ and $g$ have their non-zero entries   in, respectively, rows $1$ and $2$, and rows $1$ and $3$, then $M$ has an $r$-element circuit containing $\{e,f,g\}$. We conclude that two distinct columns of $Z$ that are not columns of $[I_r| {\bf 1}]$ must have disjoint  sets of rows containing their non-zero entries. It follows that $M$ can be obtained from a circuit $C$ containing $e$ by, for some subset $D$ of $C- e$, $2$-summing a copy of $U_{2,4}$ across each element of $D$. To see this, for each column of $Z$ with  two entries of opposite signs, add an additional column to $Z$ obtained by changing the sign of one of these entries. These added elements form the basepoints of the 2-sums.
\end{proof}

The following technical result will be helpful in proving our results on connected hyperplanes.

\begin{lemma}
\label{key}
In a simple connected matroid $M$, let $e$  be an element and $A$ be a maximal   subset of $E(M)$ that is connected, non-spanning and contains $e$. Let $C$ be a circuit of $M$ that meets both $A$ and $E(M) - A$ such that $C-A$ is  minimal. Then $A$ is a flat, $r(A \cup C) = r(M)$, the set $C-A$ is a series class of  $M|(A \cup C)$, 
$$r(M)  = r(A) + |C - A| - 1,$$ 
and one of the following holds:
\begin{itemize}
\item[(i)] $A$ is a connected hyperplane of $M$; or 
\item[(ii)] $C- A$ is a series class of $M$ with at least three elements; or 
\item[(iii)] $|C- A| \ge 3$ and $E(M) - (A \cup C)$ is non-empty.
\end{itemize}
\end{lemma}

\begin{proof}
The  minimality of $C-A$ implies that, in $M|(A \cup C)$, a circuit that meets $C-A$ must contain $C-A$, so $C - A$ is a series class. The maximality of $A$ 
implies that $A$ is a flat of $M$ and that $r(A \cup C) = r(M)$. Thus $r(M)  = r(A) + |C - A| - 1$, so $|C-A| \ge 2$. If  $|C-A| = 2$, then $A$ is a connected hyperplane of $M$. Thus we may assume that $|C-A| \ge 3$. In that case, (ii) or (iii) holds.
\end{proof}

The next result extends the theorem of Kelmans and Seymour, borrowing much from Seymour's proof.

\begin{theorem} 
\label{mainbin}
Let $e$ be an element of a  simple connected binary matroid $M$. Then 
\begin{itemize}
\item[(i)] $M$ is a circuit; or
\item[(ii)] $M$ has a connected hyperplane containing $e$; or 
\item[(iii)] $M$ has a series class of size at least three that avoids $e$.
\end{itemize}
\end{theorem}

\begin{proof} Assume that the theorem fails.    By Lemma~\ref{freedom}, $e$ is not free in $M$.  
Thus $M$ has a subset $A$ that contains a circuit containing $e$ and is maximal with respect to being connected and non-spanning. We take a circuit $C_1$ that meets both $A$ and its complement such that $|C_1 -A|$ is minimal. By Lemma~\ref{key},  $|C_1 - A| \ge 3$ and  
$E(M) - (C_1 \cup A)$ contains an element, say $x$. Moreover, for $y$ in $C_1 - A$, the set $A \cup (C_1-y)$ contains a basis $B$ of $M$. Let  $C_2 = C(x,B)$. Then $C_2$ meets $A$ otherwise, as $M$ is binary, $C_1 \triangle C_2$ is a circuit that contradicts the choice of $C_1$. Now $|C_2 -A| \ge |C_1 - A|$. Hence $C_2$ contains exactly $|C_1 - A| - 1$ elements of $C_1 - A$. Thus $C_1 \triangle C_2$ contains a circuit $D$ that contains $x$, meets $A$, and has exactly two elements not in $A$. Then $D$ contradicts the choice of $C_1$.  
\end{proof}

As an immediate consequence of this theorem, we have the following.

\begin{corollary} 
\label{mainbin2}
Let  $M$ be a   simple connected binary matroid $M$. Then 
\begin{itemize}
\item[(i)] $M$ is a circuit; or
\item[(ii)] $M$ has a connected hyperplane; or 
\item[(iii)] $M$ has at least two distinct series classes of size at least three.
\end{itemize}
\end{corollary}

The next two results of McNulty and Wu~\cite[Theorem~1.4 and Lemma~2.10]{mwu} are much more substantial extensions of  the theorem of Kelmans and Seymour. Both of these results will be used in the proof of Theorem~\ref{main_binary}.

\begin{theorem}
\label{mwu_result}
Let $M$ be a simple connected binary matroid with no cocircuits of size less than three. Then every element of $M$ is in at least two connected hyperplanes; and $M$ has at least four connected hyperplanes. 
\end{theorem}

\begin{lemma}
\label{mcnulty_wu_lemma2.10}
Let $M$ be a $3$-connected binary matroid with at least four elements. Then, for any two distinct elements $e$ and $f$ of $M$, there is a connected hyperplane containing $e$ and avoiding $f$.
\end{lemma}

McNulty and Wu~\cite[Fig. 1]{mwu} also showed that a simple connected binary matroid with no cocircuits of size less than three may have exactly four connected hyperplanes. In addition, they  noted that Joseph Bonin has pointed out that the dual of $PG(2,3)$ is a $3$-connected ternary matroid  having no connected hyperplanes. Of course, the same is true for the duals of all of the matroids $PG(r-1,3)$ with $r \ge 3$. As another example 
of a simple connected ternary matroid with no cocircuits of size less than three and no connected hyperplanes, 
 take a circuit with at least three elements and 2-sum a copy of $U_{2,4}$ across each element. Each of these examples has numerous triads. As we shall see, by confining our attention to simple connected ternary matroids having no cocircuits of size less than four, we can establish the existence of at least two connected hyperplanes.  The next result is key to proving this.

\begin{theorem} 
\label{maintern}
Let $M$ be a simple connected matroid having no cocircuits of size less than four. Assume that $M$ has no $U_{2,5}$-minor and no $U_{3,5}$-minor. Let $e$ be an element of $M$ that is not free. Then $M$ has a connected hyperplane containing $e$.
\end{theorem}

\begin{proof}
Since $e$ is not free,  $E(M)$ has a subset $A$  that contains a circuit containing $e$ and is maximal with respect to being connected and non-spanning.  Assume that the theorem fails.

As $M$ is connected, it has a circuit meeting both $A$ and its complement. Choose such a circuit $C_1$ for which $|C_1 - A|$ is a minimum. By Lemma~\ref{key}, $A$ is a flat of $M$, while $C_1 - A$ is a series class in $M|(A \cup C_1)$, and $r(A \cup C_1) = r(M)$.  Moreover, 

\begin{sublemma}
\label{rk}
$r(M) = r(A) + |C_1 - A| -1$.
\end{sublemma}

As $M$ has no cocircuits of size less than four, $|E(M) - (A \cup C_1)| \ge 2$.  Take $s$ in $C_1 - A$. Then $M$ has a basis $B$ that contains $C_1 - s$ and is contained in $A \cup (C_1 - s)$. Choose $x$ in $E(M) - (A \cup C_1)$ and let $C_2$ be $C(x,B)$.

Next we show the following.

\begin{sublemma}
\label{c2a}
If $C_2 \cap A = \emptyset$, then $|C_2| = 3$.
\end{sublemma}

Let $C_1 \cap C_2 = \{y_1,y_2,\dots,y_t\}$. By circuit elimination, $M$ has a circuit $D_i$ that contains $x$ and not $y_i$ such that $D_i \subseteq C_1 \cup x$. Then $D_i \supseteq C_1 - C_2$. Moreover, the choice of $C_1$ implies that $D_i - A = ((C_1 - A) - y_i) \cup x$. Thus $D_i = (C_1 - y_i) \cup x$. From $M|(C_1 \cup x)$, contract $C_1 - \{y_1,y_2,\dots,y_k,s\}$. The resulting matroid $N$ has ground set $\{y_1,y_2,\dots,y_k,s,x\}$ and has every subset of size $k+1$ as a circuit except possibly $\{y_1,y_2,\dots,y_k,x\}$. If a proper subset of $\{y_1,y_2,\dots,y_k,x\}$ is a circuit of $N$, then this circuit is a 
proper subset of a $(k+1)$-element circuit of $N$, a contradiction. Thus $N \cong U_{k,k+2}$. As $M$ has no $U_{3,5}$-minor and $M$ is simple, we deduce that $k = 2$, so $|C_2| = 3$. Hence (\ref{c2a}) holds.

Next we note that 

\begin{sublemma}
\label{c1a}
$|C_1 - A| \ge 4$.
\end{sublemma}

As the theorem fails, $|C_1 - A| > 2$, by \ref{rk}. Assume that $|C_1 - A| = 3$. Thus $r(M/A) = 2$. Moreover,  $|E(M/A)| \ge 5$ as $|E(M) - (A \cup C_1)| \ge 2$.  Since $A$ is a flat of $M$, the matroid $M/A$ has no loops. Suppose it has a $2$-circuit $\{u,v\}$. Then $M$ has a circuit $C'$ such that $\{u,v\} \subseteq C' \subseteq \{u,v\} \cup A$. Thus $C'$ contradicts the choice of $C_1$. We deduce that $M/A$ is simple, so $M/A$ has $U_{2,5}$ as a restriction, a contradiction. Thus \ref{c1a} holds.

\begin{sublemma}
\label{xtri}
For $x$ in $E(M) - (A \cup C_1)$, there is an element $s$ of $C_1 - A$ such that $M$ has a triangle that contains $x$ and has its other two elements in $C_1 - (A \cup s)$.
\end{sublemma}

Assume that $M$ has no such triangle. For $s$ in $C_1 - A$, let $B_s$ be  a basis of $M$ containing $C_1 - s$ and let $C_s = C(x,B_s)$. By \ref{c2a}, $C_s$ meets $A$. By the choice of $C_1$, we deduce that $C_s - A = (C_1 - A - s) \cup x$. Let $|C_1 - A| = m$ and $N' = M|(A \cup C_1 \cup x)$. Then, for every $m$-element subset $Y$ of $(C_1 - A) \cup x$, there is a circuit of $N'$ that meets $(C_1 - A) \cup x$ in $Y$ and also meets $A$. Now $r(N') = r(M) = r(A) + |C_1 - A| - 1$. Contracting $A$ from $N'$ gives a matroid of rank $m - 1$ having $m+ 1$ elements. Take an $m$-element subset $Y$ of $(C_1 - A) \cup x$. Then $Y$ contains a circuit $Y'$ of $N'/A$. Thus $Y' \cup A$ contains a circuit $Y''$ of $M$ containing $Y'$. Then $Y''$ meets $A$ otherwise 
$Y'' = Y' \subseteq Y$, a contradiction as $Y$ is independent in $M$. 
Thus, by the choice of $C_1$, we must have that $|Y'' - A| = m$. Hence $m = |Y| \ge |Y'| \ge |Y''-A| \ge m$, so 
$Y' = Y$ and $Y$ is a circuit of $N'/A$. Thus $N'/A \cong U_{m-1,m+1}$. By \ref{c1a}, $m \ge 4$, so $M$ has a $U_{3,5}$-minor, a contradiction.  Thus \ref{xtri} holds.

\begin{sublemma}
\label{no4}
$M$ does not have a $4$-element subset $X$   with exactly two elements in $C_1 - A$ and exactly two elements in $E(M) - (A \cup C_1)$ such that $M|X \cong U_{2,4}$.
\end{sublemma}

Assume that   $\{y_1,y_2,x_1,x_2\}$ is such a $4$-element subset $X$ of $E(M)$  where $\{y_1,y_2\} \subseteq C_1 - A$. Then $r(A \cup (C_1 - \{y_1,y_2\})) = r(M) - 1$ and $r(X) = 2$. Thus, in $M|(A \cup C_1 \cup \{x_1,x_2\})$, which is connected, $X$ is $2$-separating. Therefore, $M|(A \cup C_1 \cup \{x_1,x_2\})$ is the $2$-sum, with basepoint $p$ of connected matroids $M_1$ and $M_2$ with ground sets $A \cup (C_1 - \{y_1,y_2\}) \cup p$ and $X \cup p$. Since $|X \cup p| = 5$ and $M_2$ has rank $2$, we must have that $p$ is parallel to some element of $X$ otherwise $M$ has a $U_{2,5}$-minor. Thus a member of $\{y_1,y_2,x_1,x_2\}$ is in the closure of $A \cup (C_1 - \{y_1,y_2\})$. Neither $y_1$ nor $y_2$ is in this closure. If $x_1$ or $x_2$ is, then there is a circuit $D$ containing $x_i$ for some $i$ in $\{1,2\}$ such that $D \subseteq A \cup (C_1 - \{y_1,y_2\}) \cup x_i$. The choice of $C_1$ implies that $D$ does not meet $A$. Thus $D \subseteq (C_1 - A) \cup x_i$. Then, by circuit elimination, $(D \cup \{x_i,y_1,y_2\}) - x_i$ contains a circuit. But this circuit is properly contained in $C_1$, a contradiction. We conclude that \ref{no4} holds.

\begin{sublemma}
\label{no2tri}
$M$ does not have   two triangles $\{y_1,y_2,x_2\}$ and $\{y_1,y_3,x_3\}$ where $y_1, y_2$, and $y_3$ are distinct elements of $C_1 - A$, and $x_2$ and $x_3$ are distinct elements of  $E(M) - (A \cup C_1)$.
\end{sublemma}

Assume that $M$ does  have two such triangles. Then $M$ has $(C_1 - y_1) \cup x_2$ as a circuit, $C'_1$ say. Thus $\{y_2,x_2,y_3,x_3\}$ is a circuit of $M$ having exactly three elements in $C'_1 - A$.  This is the fundamental circuit of $x_3$ with respect to a basis of $M$ that contains $C'_1 - t$ where $t \in C'_1 - A - \{x_2,y_2,y_3\}$, the existence of such an element $t$ being a consequence of    \ref{c1a}. Thus, using $C'_1$ in place of $C_1$ in \ref{c2a},   we get a contradiction. Hence \ref{no2tri} holds.

By \ref{xtri}, for each element $x$ in $E(M) - (A \cup C_1)$, there is a triangle of $M$ that contains $x$ and two elements of $C_1 - A$. Moreover, by \ref{no4} and \ref{no2tri}, if $x_1$ and $x_2$ are distinct elements of $E(M) - (A \cup C_1)$, then the corresponding triangles are disjoint.

Suppose that there are exactly $k$ elements, $x_1,x_2,\dots,x_k$, in $E(M) - (A \cup C_1)$ and that the corresponding triangles are $\{x_i,y_i,z_i\}$ for $1 \le i \le k$ where $\{y_i,z_i\} \subseteq C_1 - A$. Then $E(M) - (A \cup C_1) - \{y_1,z_1\}$ has rank $r(M) - 1$. The closure of this set contains 
$\{x_2,x_3,\dots,x_k\}$. The complement of this closure is $\{x_1,y_1,z_1\}$. Therefore $M$ has a triad. This contradiction completes the proof of the theorem.
\end{proof}

\begin{corollary}
\label{ternturn}
Let $M$ be a simple connected ternary matroid having no cocircuits of size less than four. Then $M$ has at least two connected hyperplanes. 
\end{corollary}

\begin{proof}
By Lemma~\ref{freedom}(ii), since $M$ has no cocircuits of size less than four, $M$ has no free elements. Let $e$ be an element of $M$. Then, by Theorem~\ref{maintern}, $M$ has a connected hyperplane $H_e$ containing $e$. For $f$ in $E(M) - H_e$, there is a connected hyperplane $H_f$ containing $f$, so the corollary holds.
\end{proof}

In view of Theorem~\ref{mwu_result}, it is of interest to specify the minimum number of connected hyperplanes in a simple connected ternary matroid with no cocircuits of size less than four. There are infinitely many examples of such matroids with exactly four connected hyperplanes but we do not know if four is indeed the minimum number of connected hyperplanes. To get a family of examples with exactly four connected hyperplanes, first take a graph $G$ formed from 
two vertex-disjoint paths $x_1x_2\dots x_n$ and $y_1y_2\dots y_n$ for some $n \ge 1$ by adding the $n$ edges $x_iy_i$, the $n-1$ edges of the form $x_iy_{i+1}$ for $1 \le i\le n-1$, and the $n-1$ edges of the form $x_{j+1}y_{j}$ for $1 \le j\le n-1$. Then take two copies, $N_1$ and $N_2$, of $M(K_4)$, pick a point $p_i$ of $N_i$  and freely add a point $q_i$ to one of the triangles of $N_i$ not containing $p_i$. Finally, take the parallel connection of $N_1$ and $M(G)$ with respect to the basepoints $p_1$ and $x_1y_1$, and then attach $N_2$ to the resulting matroid via parallel connection with respect to the basepoints $x_ny_n$ and $p_2$. The resulting simple connected ternary matroid has $5n+8$ elements, rank $2n+3$, and has no cocircuits of size less than four. It also has exactly four connected hyperplanes.

\section{Induced-restriction-minimal non-$GF(2)$-comatroids}
\label{bin}

An \textbf{induced-restriction-minimal non-$GF(q)$-comatroid} is a $GF(q)$-representable matroid $M$ that is not a $GF(q)$-comatroid such that every proper flat of $M$  is a $GF(q)$-comatroid.  The collection of such matroids $M$ will be denoted by  $\mathcal{M}_q$.  Clearly, 
$M^c \in \mathcal{M}_q$ for every matroid $M$ in $\mathcal{M}_q$.    This section begins with some preliminary results that will be used in the proofs of the main theorems. It concludes with  proofs of Theorem~\ref{main_binary} and Corollary~\ref{cor_binary}.

\begin{lemma}
\label{enough_elements_imply_connected}
For $q$ in $\{2,3\}$, let $X$ be a subset of $PG(r-1,q)$ having at least $q^{r-1}+1$ elements. Then the matroid $PG(r-1,q)|X$ is connected and has rank $r$.
\end{lemma}

\begin{proof}
Observe that $X$ has more elements than a hyperplane of $PG(r-1,q)$, so $PG(r-1,q)|X$ has rank $r$. Assume that $PG(r-1,q)|X$ is disconnected. Then, for some $j$ with $1 \leq j \leq r-1$, the matroid $PG(r-1,k)|X$ is contained in $PG(j-1,q) \oplus PG(r-j-1,q)$. Thus $|X| \leq \frac{q^j - 1}{q-1} + \frac{q^{r-j} -1}{q-1}  = \frac{q^j + q^{r-j} - 2}{q-1}$. This function of $j$ is maximized when $j$ is $1$ or $r-1$, so $|X| \leq q^{r-1}$, a contradiction.
\end{proof}

\begin{lemma}
\label{parallel_connection_complement}
For $q$ in $\{2,3\}$, let $N$ be the parallel connection of $PG(j-1,q)$ and $PG(k-1,q)$ where $2 \leq j \leq k$ and $k \geq 3$. Then the complement $N^c$ of $N$ has rank equal to $r(N)$.
\end{lemma}

\begin{proof}
Assume that $r(N^c) < r(N)$. Then $N$ has $AG(r(N)-1,q)$ as a restriction. Now $AG(r(N)-1,q)$ is $3$-connected since $r(N) \geq 4$ so $N$ is $3$-connected, a contradiction.
\end{proof}

The next result is  from \cite{oxwu}.

\begin{theorem}
\label{oxwu_result}
Let $n$ be an integer exceeding one and $X$ and $Y$ be subsets of the ground set of a matroid $M$. Suppose $M|X$ and $M|Y$ are both vertically $n$-connected and $r(X) + r(Y) - r(X \cup Y) \geq n-1$. Then $M|(X \cup Y)$ is vertically $n$-connected. 
\end{theorem}

The following is a straightforward consequence of Proposition \ref{sum_of_connectivities}.

\begin{lemma}
\label{red_or_green_connected}
Let $r \geq 4$. For an arbitrary prime power $q$, color the elements of $PG(r-1,q)$ red or green. Then either $PG(r-1,q)|G$ or $PG(r-1,q)|R$ is connected of rank $r$.
\end{lemma}

Recall that, for a flat  $F$ in a  simple $GF(q)$-represented matroid $M$, we write $F^c$ for the matroid $(M|F)^c$.

\begin{lemma}
\label{cosimple_case}
Let $M$ be a matroid in $\mathcal{M}_2$ such that $r(M) \geq 5$ and $M$ has a $2$-cocircuit. Then $M$ is isomorphic to a circuit or to $P(U_{3,4}, U_{3,4})$.
\end{lemma}

\begin{proof} 
Assume that the result fails. Since $M$ has a $2$-cocircuit, it has a maximal non-trivial series class $S$. Thus $M= M_1 \oplus_2 M_2$ where $M_2$ is a circuit with ground set $S \cup p$,  and $p$ is the basepoint of the $2$-sum. If $p$ is parallel to an element $s$ in $M_1$, then we move $s$ into $M_2$ so that it become parallel to $p$ there. 

Suppose that $p$ is free in $M_1$. Then, by Lemma~\ref{freedom}(i), $M_1$ is a circuit. As $M$ is not a circuit and $r(M) \ge 5$, we deduce that the element $s$ exists. Thus $M$ is the parallel connection of two circuits. By Corollary~\ref{cct}, neither of these circuits has more than four elements. Hence $M \cong P(U_{3,4}, U_{3,4})$, a contradiction. We deduce that $p$ is not free in $M_1$. Thus $M_1$ has a non-spanning circuit $C_p$ that contains $p$. 
 If $r((C_p \cup S) - p) \ge 4$, then the closure,  $F$, of $(C_p \cup S) - p$ is a connected proper flat in $M$, Moreover, by Lemmas~\ref{sum_of_connectivities} and  \ref{parallel_connection_complement}, $F^c$ is also connected of rank $r(F)$. This contradicts the minimality of $M$. We deduce that $r((C_p \cup S) - p) = 3$. Hence every non-trivial series class of $M$ has exactly two elements. Now, by Theorem~\ref{mainbin}, as $M_1$ is not a circuit and does not have a series class of size at least three avoiding $p$, it has a connected hyperplane $H$ containing $p$. Then ${\rm cl}_M((H \cup S) - p)$ is a connected proper flat, $F$, of $M$ of rank $r(M) - 1$. As above, $F^c$ is connected of rank $r(F)$, a contradiction. 
\end{proof}

\begin{lemma}
\label{M_3_connected}
Let $M$ be a matroid in $\mathcal{M}_2$ such that $r(M) \geq 5$. Then 
\begin{itemize}
\item[(i)] $M$ is a circuit; or 
\item[(ii)] $M \cong P(U_{3,4}, U_{3,4})$; or 
\item[(iii)] $M$ is $3$-connected. 
\end{itemize}
\end{lemma}

\begin{proof}
Assume that neither (i) nor (ii) holds. Then, by Lemma~\ref{cosimple_case}, we may assume that $M$ is cosimple. 
Suppose that $M$ is not $3$-connected. Then $M = M_1 \oplus_2 M_2$ where $r(M_1) \geq r(M_2)$ and one of $M_1$ and $M_2$ may have an element parallel to the basepoint $p$ of the $2$-sum. When this occurs, we may assume, by moving an element from $M_2$ to $M_1$ if needed, that the element is parallel to $p$ in $M_2$. Since $M$ is cosimple, neither $M_1$ nor $M_2$ is either a circuit or a circuit with an element parallel to $p$. Hence $r(M_2) \ge 3$. As $M_2$ is not a circuit, by Lemma~\ref{freedom}, $M_2$ has a non-spanning circuit $C_p$ containing $p$. Then the closure $F$ of 
$(E(M_1) \cup C_p) - p$ is a connected proper flat of $M$. By Lemmas~\ref{sum_of_connectivities} and \ref{parallel_connection_complement}, $F^c$ is connected of rank $r(F)$, a contradiction. 
\end{proof}

The next result shows that a matroid $M$ in $\mathcal{M}_2$ such that neither $M$ nor $M^c$ is a circuit has rank at most five. 

\begin{theorem}
\label{main_theorem_rank_bound}
Let $M$ be a matroid in $\mathcal{M}_2$ such that $r(M) \geq 6$. Then $M$ or $M^c$ is a circuit.
\end{theorem}

\begin{proof}
 Let $P_{r(M)}$ denote the binary projective geometry of rank $r(M)$ such that the set $G$ of green elements of $P_{r(M)}$ corresponds to $M$ and the set $R$ of red elements of $P_{r(M)}$ corresponds to $M^c$.
Observe that, for each projective flat $F$ of $P_{r(M)}$ with $4 \leq r(F) < r(M)$, it follows by Lemma \ref{red_or_green_connected} and the minimality of $M$ that exactly one of $F|R$ and $F|G$ is connected of rank $r(F)$. We call $F$ \textbf{red} or \textbf{green} depending on whether $F|R$ or $F|G$ is connected of rank $r(F)$. We may assume that both $M$ and $M^c$ are cosimple otherwise we have our result by Lemma \ref{cosimple_case}. Let $F$ be a rank-$(r-2)$ flat of $P_{r(M)}$. Note that $F$ is contained in exactly three hyperplanes, say $H_1, H_2$, and $H_3$ of $P_{r(M)}$. We note the following.

\begin{sublemma}
\label{hsize2}
At least two of $H_1$, $H_2$, and $H_3$ have the same color as $F$.
\end{sublemma}

Suppose that $F$ is green and assume that $H_1$ and $H_2$ are red. It follows that each of $H_1 - F$ and $H_2-F$ contains at most one green element and so the green elements in $(H_1 \cup H_2) - F$ form a cocircuit of $M$ with at most two elements, a contradiction. 
Similarly, if $F$ is red, we get a cocircuit of $M^c$ of size at most two,  a contradiction.

Now let $G_1$ and $R_1$ be the sets of green and red hyperplanes, respectively, of $P_{r(M)}$. We note the following.

\begin{sublemma}
\label{nbr}
At most one of the rank-$(r-2)$ projective flats contained in a projective hyperplane $H$ has a color different from that of $H$.
\end{sublemma}

Observe that if two rank-$(r-2)$ projective flats contained in $H$ have the same color, then, by Theorem \ref{oxwu_result}, their join is colored the same as the two flats, a contradiction. Thus \ref{nbr} holds.

Let $G_2$ and $R_2$ be the sets of green and red projective flats of $P_{r(M)}$ of rank $r-2$. We consider the bipartite graph $B$ with vertex sets $G_1 \cup R_1$ and $G_2 \cup R_2$ such that a vertex $X$ in $G_1 \cup R_1$ is adjacent to a vertex $Y$ in $G_2 \cup R_2$ if $Y \subseteq X$. We count the number of \textbf{cross edges} of this graph, that is, the $G_1R_2$-edges and $G_2R_1$-edges. By \ref{nbr}, the number of $G_1R_2$-edges is at most $|G_1|$, and the number of $G_2R_1$-edges is at most $|R_1|$. Observe that each pair $\{H_G, H_R\}$ where $H_G \in G_1$ and $H_R \in R_1$ corresponds to either a $G_1R_2$- edge or a $G_2R_1$-edge $e$ depending on whether $H_G \cap H_R$ is red or green. Note that there is a third projective hyperplane $H'$ such that $H' \cap H_G = H' \cap H_R$ and, by \ref{hsize2}, has the same color as $H_G \cap H_R$. Observe that either $\{H', H_G\}$ or $\{H',H_R\}$ corresponds to the cross edge $e$ depending on whether $H_G \cap H_R$ is red or green. Hence  each cross edge corresponds to exactly two pairs $\{H_G,H_R\}$ where $H_G \in G_1$ and $H_R \in R_1$. As the number of cross edges is bounded above by $|G_1|+|R_1|$ and below by $\frac{1}{2}|G_1||R_1|$, we have $\frac{1}{2}|G_1||R_1| \leq |G_1|+|R_1|$. We may assume that $|R_1| \ge |G_1|$. Then $|G_1| \leq  \frac{2|G_1|}{|R_1|} + 2 \leq 3$, a contradiction to Theorem \ref{mwu_result}. 
\end{proof}

It remains to determine the members of $\mathcal{M}_2$ of rank $4$ or $5$. The next lemma takes care of the rank-$4$ case.

\begin{lemma}
\label{bin_rk_four}
A rank-$4$ binary matroid $M$ is a  member of $\mathcal{M}_2$ if and only if $M$ or $M^c$ is the cycle matroid of one of the six graphs in Figure~\ref{main_M_2}.
\end{lemma}

\begin{proof}
First assume that $M \in \mathcal{M}_2$.  We may assume that $|E(M)| \le |E(M^c)|$, so $|E(M)| \le 7$. If $M$ has a $5$-circuit, then $M$ is a $5$-circuit or a $1$- or $2$-element extension thereof. One can now check that $M$ is the cycle matroid of one of the graphs on the first line 
of Figure~\ref{main_M_2}. We may now assume that $M$ has no $5$-circuits. Thus $|E(M)|$ is $6$ or $7$. If $|E(M)| = 6$, then $M^*$ is connected of rank two, so $M$ is the cycle matroid of $K_{2,3}$. Finally, if $|E(M)| = 7$, then $M$ is the cycle matroid of the last graph in Figure~\ref{main_M_2}. The proof of the converse is immediate as every rank-$3$ binary matroid is a binary comatroid.
\end{proof}

The following result  from \cite{marowh} will be used to simplify the computational task of finding the rank-$5$ members of $\mathcal{M}_2$. The matroids in this theorem will only appear  in the proof of Lemma~\ref{rank_five_binary} and they will be defined there.

\begin{theorem}
\label{mayhew_result}
An internally $4$-connected binary matroid has no $M(K_{3,3})$- minor if and only if it is

\begin{enumerate}[label=(\roman*)]
    \item cographic; or
    \item isomorphic to a triangular or triadic M\"{o}bius matroid; or
    \item isomorphic to one of $18$ sporadic matroids of rank at most $11$.
\end{enumerate}

\end{theorem}

\begin{lemma}
\label{rank_five_binary}
Let $M$ be a matroid in $\mathcal{M}_2$ such that $r(M)=5$. Then $M$ or $M^c$ is not cosimple.
\end{lemma}

\begin{proof}

Assume that $M$ and $M^c$ are cosimple.  Then, by Lemma~\ref{M_3_connected}, both $M$ and $M^c$ are $3$-connected.  By Lemma \ref{red_or_green_connected}, for every hyperplane $H$ of $P_5$, exactly one of $H|G$ or $H|R$ is connected. We call the  $H$ red or green depending on whether $H|R$ or $H|G$ is connected of rank four. We first show that

\begin{sublemma}
\label{binary_towards_int_4_conn_sublemma}
$E(M)$ has no set $X$ of rank $3$ such that $r(E(M)-X) = 4$.
\end{sublemma}

Denote $E(M)-X$ by $Y$ and assume that $r(X)=3$ and $r(Y)=4$. Let $Y_P$ and $X_P$ denote the projective flats spanned by $Y$ and $X$, respectively. Observe that $Y_P \cap X_P$ is a projective line, say, $L = \{x,y,z\}$. As $M$ has no $2$-cocircuits, it follows that $E(M)-Y_P$ has at least three elements, including say $b_1,$ $b_2,$ and $b_3$ such that $\{x,b_1,b_2\}$ is a projective line, say $L_1$. Let $M'$ be the matroid obtained from $M|Y_P$ by adding $x,y,$ and $z$ if they are not already  in $E(M)$. Note that, for $k$ in $\{1,2\}$, a $k$-separation of $M'$ induces a $k$-separation of $M$ and therefore $M'$ is $3$-connected. By Lemma \ref{mcnulty_wu_lemma2.10}, $M'$ has a connected hyperplane $H$ that contains $x$ but not $y$ or $z$. Observe that either $P(H,L_1)$ or $P(H,L_1)\ba x$ is a connected hyperplane $H'$ of $M$ depending on whether or not $x$ is an element of $E(M)$. By Proposition \ref{sum_of_connectivities} and Lemma \ref{parallel_connection_complement}, it follows that the complement of $H'$ is connected of rank four, a contradiction.

\begin{sublemma}
\label{binary_hyperplane_bound_sublemma}
A connected hyperplane of $M$ has at least seven elements.
\end{sublemma}

Suppose such a connected hyperplane has at most six elements. Then its complement in $PG(3,2)$ has at least nine elements. By Lemma \ref{enough_elements_imply_connected}, this complement is connected of rank four, a contradiction. Thus \ref{binary_hyperplane_bound_sublemma} holds.

By  \ref{binary_towards_int_4_conn_sublemma}, it follows that $M$ is internally $4$-connected and has no triads. By Theorem \ref{mwu_result}, $M$ has a connected hyperplane, so $|E(M)| \geq 11$ by \ref{binary_hyperplane_bound_sublemma}. 

Suppose that $M$ has no $M(K_{3,3})$-minor. By Theorem \ref{mayhew_result}, we first suppose that $M$ is cographic and therefore $E(M) \leq 12$ \cite{jaegar}. Since $M$ has no cocircuits of size less than four, it follows that every hyperplane of $M$ has at most eight elements. Therefore, for a connected hyperplane $H$ of $M$, by \ref{binary_hyperplane_bound_sublemma}, $|H|$ is either seven or eight. It follows that  $H^c$ has seven or eight elements. As this complement is either disconnected or has rank at most three, it is either $F_7$ or $F_7 \oplus U_{1,1}$. This implies that $H$ has $F_7^*$ as a restriction. Thus $M$ is not regular, a contradiction.

Next suppose that $M$ is a triangular or a triadic M\"{o}bius matroid. Since $M$ has no triads, $M$ is the rank-$5$ triangular M\"{o}bius matroid, $\Delta_5$, and has the  reduced representation \cite{marowh} shown on the left of Figure~\ref{d5m511}. Observe that $\{e, j, k, l, m\}$ is a connected hyperplane of $M$ of size five, a contradiction to \ref{binary_hyperplane_bound_sublemma}. We may now assume that $M$ is a rank-$5$ sporadic matroid, so $M$ is isomorphic to a matroid in  $\{M_{5,11}, T_{12}/e,  M_{5,12}^a, M_{5,12}^b,  M_{5,13}\}$ \cite{marowh}. Since $M_{5,11}$ has a triad, $M$ is not isomorphic to $M_{5,11}$. When $M$ is isomorphic to $T_{12}/e$, it has the   representation shown on the right of 
Figure~\ref{d5m511}. Then $\{f,g,h,i,j\}$ is a connected hyperplane of $M$, a contradiction to \ref{binary_hyperplane_bound_sublemma}.

\begin{figure}[ht]
\begin{minipage}[b]{0.45\linewidth}
\centering\[
\bordersquare
{
   & f & g & h & i & j & k & l & m\cr
a & 1 & 0 & 0 & 0 & 1 & 0 & 0 & 1  \cr
b & 0 & 1 & 0 & 0 & 1 & 1 & 0 & 0  \cr
c & 0 & 0 & 1 & 0 & 0 & 1 & 1 & 0 \cr
d & 0 & 0 & 0 & 1 & 0 & 0 & 1 & 1 \cr
e & 1 & 1 & 1 & 1 & 0 & 0 & 0 & 1 \cr
}
 \]
 \end{minipage} 
\begin{minipage}[b]{0.45\linewidth}
\centering
\[
\bordersquare
{
   & f & g & h & i & j & k \cr
a & 1 & 0 & 0 & 0 & 1 & 1 \cr
b & 1 & 1 & 0 & 0 & 0 & 1  \cr
c & 0 & 1 & 1 & 0 & 0 & 1 \cr
d & 0 & 0 & 1 & 1 & 0 & 1 \cr
e & 0 & 0 & 0 & 1 & 1 & 1 \cr
}
\]
\end{minipage}
\caption{$\Delta_5$ and $T_{12}/e$.}
\label{d5m511}
\end{figure}

If $M$ is isomorphic to $M_{5,12}^a$, then $M$ has the   representation shown on the left of Figure~\ref{m512}. Then $M$ has  $\{f,g,h,i,j,l\}$ as a connected hyperplane of $M$,  contradicting \ref{binary_hyperplane_bound_sublemma}. 
Similarly, if $M$ is isomorphic to $M_{5,12}^b$, then $M$ has the  representation shown on the right of Figure~\ref{m512}. Then $\{f,g,h,i,j,l\}$ is a connected hyperplane of $M$, again contradicting \ref{binary_hyperplane_bound_sublemma}.

\begin{figure}[hb]
\begin{minipage}[b]{0.45\linewidth}
\centering\[
\bordersquare
{
  &  f & g & h & i & j & k & l \cr
a & 1 & 0 & 0 & 0 & 1 & 1 & 0 \cr
b & 1 & 1 & 0 & 0 & 0 & 0 & 1 \cr
c & 0 & 1 & 1 & 0 & 0 & 1 & 0 \cr
d & 0 & 0 & 1 & 1 & 0 & 0 & 1 \cr
e & 0 & 0 & 0 & 1 & 1 & 1 & 0 \cr
}
 \]
 \end{minipage} 
\begin{minipage}[b]{0.45\linewidth}
\centering
\[
\bordersquare
{
  &  f & g & h & i & j & k & l \cr
a & 1 & 0 & 0 & 0 & 1 & 1 & 0 \cr
b & 1 & 1 & 0 & 0 & 0 & 1 & 1 \cr
c & 0 & 1 & 1 & 0 & 0 & 1 & 1\cr
d & 0 & 0 & 1 & 1 & 0 & 1 & 1 \cr
e & 0 & 0 & 0 & 1 & 1 & 1 & 1 \cr
}
\]
\end{minipage}
\caption{$M_{5,12}^a$ and $M_{5,12}^b$.}
\label{m512}
\end{figure}

We may now assume that $M$ is isomorphic to $M_{5,13}$ and therefore has the  representation in Figure~\ref{m513}. Observe that $\{a,b,d,e,f,i,j\}$ is a connected hyperplane, $H$, of $M$ such that $H^c$ is also connected of rank $4$, a contradiction. We conclude that $M$ is not one of the five rank-$5$ sporadic matroids.

\begin{figure}
\[
  \bordersquare
  {
 & f & g & h & i & j & k & l & m \cr
a & 1 & 0 & 0 & 0 & 1 & 1 & 0 & 1 \cr
b & 1 & 1 & 0 & 0 & 0 & 1 & 1 & 1 \cr
c & 0 & 1 & 1 & 0 & 0 & 1 & 1 & 1 \cr
d & 0 & 0 & 1 & 1 & 0 & 1 & 1 & 1 \cr
e & 0 & 0 & 0 & 1 & 1 & 1 & 1 & 0 \cr
}
 \]
 \caption{$M_{5,13}$.}
\label{m513}
 \end{figure}

We may now assume that  $M$ has an $M(K_{3,3})$-minor and so $M$ is an extension of $M(K_{3,3})$. By symmetry, $M^c$ is also an extension of $M(K_{3,3})$. Since $P_5$ has $31$ hyperplanes and $M(K_{3,3})$ has six connected hyperplanes, we deduce that 

\begin{sublemma}
\label{green_hyperplane_bound}
$M$ has at most $25$ connected hyperplanes.
\end{sublemma}

Figure~\ref{k33} shows the vertex-edge incidence matrix of $K_{3,3}$, which  is a binary representation for $M(K_{3,3})$. Although $r(M(K_{3,3}))=5$, we use this representation because it displays the symmetries of $M(K_{3,3})$  well. The $P_5$ into which $M$ is embedded is spanned by $\{a,b,c,d,e,f,g,h,i\}$.

\begin{figure}[b]
\[
\begin{blockarray}{cccccccccc}
 & a & b & c & d & e & f & g & h & i  \\
\begin{block}{c(ccccccccc)}
 v_1 & 1 & 1 & 1 & 0 & 0 & 0 & 0 & 0 & 0   \\
 v_2 & 0 & 0 & 0 & 1 & 1 & 1 & 0 & 0 & 0   \\
 v_3 & 0 & 0 & 0 & 0 & 0 & 0 & 1 & 1 & 1  \\
 v_4 & 1 & 0 & 0 & 1 & 0 & 0 & 1 & 0 & 0  \\
 v_5 & 0 & 1 & 0 & 0 & 1 & 0 & 0 & 1 & 0  \\
 v_6 & 0 & 0 & 1 & 0 & 0 & 1 & 0 & 0 & 1  \\
\end{block}
\end{blockarray}
 \]
 \caption{The vertex-edge incidence matrix of $K_{3,3}$.}
\label{k33} 
\end{figure}

\begin{figure}[h]
\center
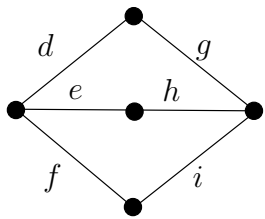
\caption{The labelled $K_{2,3}$ corresponding to the hyperplane $H_1$.}
\label{k23fig}
\end{figure}

For $1 \leq i \leq 6$, let $H_i$ be the connected hyperplane of $M(K_{3,3})$ that is complementary to the vertex bond of $v_i$ in $K_{3,3}$, and let $H'_i$ be the hyperplane of $P_5$ spanned by $H_i$. As $H'_1|G$ is an extension of $M(K_{2,3})$, it follows that $H'_1|R$ is a restriction of the complement of $M(K_{2,3})$ in $P_4$. This complement is isomorphic to $P(F_7,U_{2,3})$, where $p$ is the basepoint of the parallel connection. 
The $K_{2,3}$ corresponding to $H_1$ is shown in Figure~\ref{k23fig}. The matroid $P(F_7,U_{2,3})$ that is the complement of this $M(k_{2,3})$ is labelled as in Figure~\ref{f7fig}. Here elements are labelled by the corresponding vectors. 
Because $H'_1|R$ is not connected of rank $4$, it is isomorphic to a restriction of either $U_{2,3} \oplus U_{2,3}$ or $F_7 \oplus U_{1,1}$. 
Assume the former. Then the red elements of $H'_1$ are contained in the $2$-separating triangle in $P(F_7, U_{2,3})$ and one of the four triangles of $F_7$ that avoid $p$, where $p$ corresponds to the vector $d+g$. Thus we have the following four cases:

\begin{enumerate}[label=(\roman*)]
    \item $e+g, e+i,$ and $d+i$ are green;
    \item $e+g, g+i,$ and $e+f$ are green;
    \item $d+e, g+i,$ and $e+i$ are green;
    \item $d+e, d+i,$ and $e+f$ are green.
\end{enumerate}
By permuting the vertices $v_4, v_5$, and $v_6$, we see that the last three cases are symmetric. Thus $M$ is an extension of one of the two matroids whose representations are shown in  Figure~\ref{m2}.

\begin{figure}[h]
\center
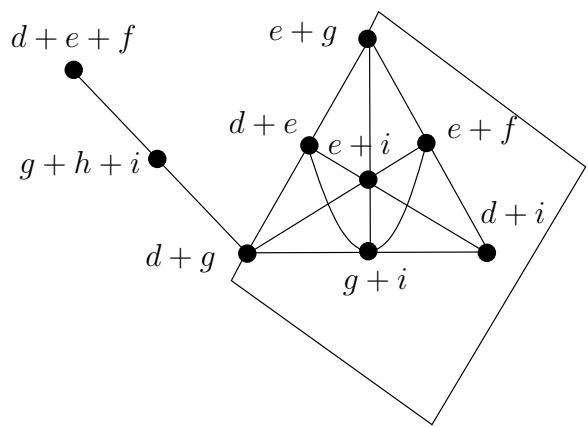
\caption{The labelled  $P(F_7, U_{2,3})$ corresponding to the complement of $H_1$ in $P_4$.}
\label{f7fig}
\end{figure}

\begin{figure}
\[
  \bordersquare
  {
 &a & b & c & d & e & f & g & h & i & g+e & d+i & i+e  \cr
  &1 & 1 & 1 & 0 & 0 & 0 & 0 & 0 & 0 & 0 & 0 & 0  \cr
  &0 & 0 & 0 & 1 & 1 & 1 & 0 & 0 & 0 & 1 & 1 & 1  \cr
  &0 & 0 & 0 & 0 & 0 & 0 & 1 & 1 & 1 & 1 & 1 & 1  \cr
  &1 & 0 & 0 & 1 & 0 & 0 & 1 & 0 & 0 & 1 & 1 & 0  \cr
  &0 & 1 & 0 & 0 & 1 & 0 & 0 & 1 & 0 & 1 & 0 & 1   \cr
  &0 & 0 & 1 & 0 & 0 & 1 & 0 & 0 & 1 & 0 & 1 & 1  \cr
}
 \]

\[
  \bordersquare
  {
& a & b & c & d & e & f & g & h & i & g+e & g+i & e+f   \cr
 & 1 & 1 & 1 & 0 & 0 & 0 & 0 & 0 & 0 & 0 & 0 & 0  \cr
  &0 & 0 & 0 & 1 & 1 & 1 & 0 & 0 & 0 & 1 & 0 & 0   \cr
  &0 & 0 & 0 & 0 & 0 & 0 & 1 & 1 & 1 & 1 & 0 & 0   \cr
 & 1 & 0 & 0 & 1 & 0 & 0 & 1 & 0 & 0 & 1 & 1 & 0   \cr
  &0 & 1 & 0 & 0 & 1 & 0 & 0 & 1 & 0 & 1 & 0 & 1  \cr
 & 0 & 0 & 1 & 0 & 0 & 1 & 0 & 0 & 1 & 0 & 1 & 1   \cr
}
 \]
 \caption{Two extensions of $M(K_{3,3})$.}
 \label{m2}
 \end{figure}
 
 
 \begin{algorithm}
\renewcommand{\thealgorithm}{}
\label{pseudocode}
\caption{Counting hyperplanes of $M$ and $M^c$}

\begin{algorithmic}
\REQUIRE Input a simple binary matroid $N$ of rank five
\STATE Set $i \leftarrow 0$, $j \leftarrow 0$

\FOR{a subset $S$ of $P_5 - E(N)$}

\STATE Set $i \leftarrow 0$, $j \leftarrow 0$

\STATE Set $M = P_5|(E(N) \cup S)$

\STATE $i \leftarrow $  number of connected hyperplanes of $M$.

\IF{$i < 26$}

\STATE $j \leftarrow$ number of connected hyperplanes of $M^c$.

\ENDIF

\ENDFOR

\IF {$i+j < 32$}

\STATE print $M$
\ENDIF

\end{algorithmic}
\end{algorithm}


Using SageMath \cite{sage}, we apply the given hyperplane-counting algorithm  to the   two matroids whose representations are given in Figure~\ref{m2}. This shows that, for every extension of these matroids, either the number of green hyperplanes exceeds $25$, a contradiction to \ref{green_hyperplane_bound}; or the sum of the number of red and green hyperplanes exceeds $31$, the number of hyperplanes of $P_5$, and again we have  a contradiction. Note that, when we run the above algorithm with $|S| = 10$, we do not obtain any matroids. Thus the search can be restricted to sets $S$ with at most ten elements. 
 
Next suppose that $H'_1|R$ is a restriction of a copy of $F_7 \oplus U_{1,1}$. First we assume that these red elements are contained in the union of the $2$-separating triangle of $P(F_7,U_{2,3})$ with another triangle through $p$ and one further point, $z$. Although there are three such lines through $p$ and four choices for $z$ for each, permuting $v_4, v_5$, and $v_6$ reduces these twelve cases to the  following two cases:

\begin{enumerate}[label=(\roman*)]
    \item $e+i, e+f,$ and $g+i$ are green;
    \item $e+i, e+f$ and $d+i$ are green.
\end{enumerate}
Thus $M$ is an extension of one of the two matroids whose representations are shown in  Figure~\ref{extra}. Again using SageMath \cite{sage} and applying the given hyperplane-counting algorithm  to these   two matroids, we see that, for every extension of these matroids, either the number of green hyperplanes exceeds $25$,  or the sum of the number of red and green hyperplanes exceeds $31$, so we obtain a contradiction. As in the previous case check,    we find that we  can   restrict the search to sets $S$ with at most ten elements. 

\begin{figure}[t]
\[
  \bordersquare
  {
 &a & b & c & d & e & f & g & h & i & e+ i & e+f & g+i  \cr
  &1 & 1 & 1 & 0 & 0 & 0 & 0 & 0 & 0 & 0 & 0 & 0  \cr
  &0 & 0 & 0 & 1 & 1 & 1 & 0 & 0 & 0 & 1 & 0 & 0  \cr
  &0 & 0 & 0 & 0 & 0 & 0 & 1 & 1 & 1 & 1 & 0 & 0  \cr
  &1 & 0 & 0 & 1 & 0 & 0 & 1 & 0 & 0 & 0 & 0 & 1  \cr
  &0 & 1 & 0 & 0 & 1 & 0 & 0 & 1 & 0 & 1 & 1 & 0   \cr
  &0 & 0 & 1 & 0 & 0 & 1 & 0 & 0 & 1 & 1 & 1 & 1  \cr
}
 \]

\[
  \bordersquare
  {
& a & b & c & d & e & f & g & h & i & e+i & e+f & d+i   \cr
 & 1 & 1 & 1 & 0 & 0 & 0 & 0 & 0 & 0 & 0 & 0 & 0  \cr
  &0 & 0 & 0 & 1 & 1 & 1 & 0 & 0 & 0 & 1 & 0 & 1   \cr
  &0 & 0 & 0 & 0 & 0 & 0 & 1 & 1 & 1 & 1 & 0 & 1   \cr
 & 1 & 0 & 0 & 1 & 0 & 0 & 1 & 0 & 0 & 0 & 0 & 1   \cr
  &0 & 1 & 0 & 0 & 1 & 0 & 0 & 1 & 0 & 1 & 1 & 0  \cr
 & 0 & 0 & 1 & 0 & 0 & 1 & 0 & 0 & 1 & 1 & 1 & 1   \cr
}
 \]
 \caption{Two more extensions of $M(K_{3,3})$.}
 \label{extra}
 \end{figure}

We may now assume that, for $1 \leq i \leq 6$, each $H'_i|R$ is a restriction of $F_7 \oplus U_{1,1}$ where the coloop in $F_7 \oplus U_{1,1}$ is one of the elements $d+e+f$ or $g+h+i$. Then,  for the red elements of each of $H'_1, H'_2, H'_3$ to behave in this way, we must have at least two points in $\{a+b+c, d+e+f, g+h+i\}$ colored green. Similarly, for $H'_4, H'_5, H'_6$, we must have at least two points in $\{a+d+g, b+e+h, c+f+i\}$ colored green. Using symmetry, we may assume that $a+b+c, d+e+f, a+d+g, b+e+h$ are green. Thus $M$ is an extension of the  matroid whose representation is shown in Figure~\ref{f77}. Using SageMath \cite{sage}, we see that this matroid has exactly $27$ connected hyperplanes, a contradiction. Hence the  lemma holds.
\end{proof}

\begin{figure}[b]
\[
  \bordersquare
  {
& a & b & c & d & e & f & g & h & i &     &  &   &   \cr
&  1 & 1 & 1 & 0 & 0 & 0 & 0 & 0 & 0 & 1 & 0 & 1 & 1  \cr
&  0 & 0 & 0 & 1 & 1 & 1 & 0 & 0 & 0 & 0 & 1 & 1 & 1  \cr
&  0 & 0 & 0 & 0 & 0 & 0 & 1 & 1 & 1 & 0 & 0 & 1 & 1   \cr
 & 1 & 0 & 0 & 1 & 0 & 0 & 1 & 0 & 0 & 1 & 1 & 1 & 0   \cr
&  0 & 1 & 0 & 0 & 1 & 0 & 0 & 1 & 0 & 1 & 1 & 0 & 1   \cr
 & 0 & 0 & 1 & 0 & 0 & 1 & 0 & 0 & 1 & 1 & 1 & 0 & 0   \cr
}
 \]
 \caption{$M$ is an extension of this matroid whose last four columns are $a+b+c$, $d+e+f$, $a+d+g$, and $b+e+h$.}
 \label{f77}
 \end{figure}

We can now prove our main results for binary comatroids.

\begin{proof}[Proof of Theorem~\ref{main_binary}.] 
Let $M$ be a binary comatroid. 
Then, by Theorem~\ref{equiv}, every flat of each of $M$ and $M^c$ is a binary comatroid. Thus, by Corollary~\ref{cct}, none of these flats is a circuit of size exceeding five. By Proposition~\ref{sum_of_connectivities} and Lemma~\ref{parallel_connection_complement}, none of the flats is isomorphic to $P(U_{3,4},U_{3,4})$. Finally, by  Lemma~\ref{bin_rk_four}, none of the flats is isomorphic to the cycle matroid of one of the six graphs in  
Figure~\ref{main_M_2}. 

Conversely, suppose that $M$ is a binary matroid that is not a comatroid. Then $M$ has a flat $N$ that is a member of $\mathcal{M}_2$. By 
Lemma~\ref{comat_small_rank}, $r(N) \ge 4$. By Lemma~\ref{bin_rk_four}, if $r(N) = 4$, then $N$ or $N^c$ is isomorphic to the cycle matroid of one of the six graphs in  Figure~\ref{main_M_2}. Thus $M$ or $M^c$ has a flat that is isomorphic to one of these cycle matroids. We may now assume that $r(N) \ge 5$. If $r(N) \ge 6$, then, by Theorem~\ref{main_theorem_rank_bound}, $N$ or $N^c$ is a circuit, so $M$ or $M^c$ has a circuit as a flat. Finally, if $r(N) = 5$, then, by  Lemmas~\ref{M_3_connected} and \ref{rank_five_binary}, we get that $M$ or $M^c$ has as a flat either a circuit or $P(U_{3,4},U_{3,4})$. 
\end{proof}

 Because we are only dealing with simple matroids, in the next proof and from now on, whenever we write $M/e$, we shall mean ${\rm si}(M/e)$.

\begin{proof}[Proof of Corollary~\ref{cor_binary}.] 
By Lemma~\ref{comat_small_rank}, every binary matroid of rank at most three is a comatroid. Thus, 
in view of Theorem~\ref{main_binary}, it suffices to prove that $M$ is an induced-minor-minimal binary non-comatroid when $M^c$ is either a circuit of size at least six or is isomorphic to $P(U_{2,3},U_{2,3})$. Consider the first case. Since both $M$ and $M^c$ are connected, $M$ is not a binary comatroid. 
Observe that, for any proper flat $F$ of $M$, the matroid $(M|F)^c$ is free and so, by Lemma~\ref{closure_under_flats}, $M|F$ is a binary comatroid. Note that, in view of Lemma \ref{closure_under_flats} and Proposition \ref{comat_contract}, it is enough to show that $M/e$ is a binary comatroid for all $e$ in $E(M)$. Because there is at most one red element on any line though $e$, we see that $(M/e)^c$ has a most one point, so $M/e$ is a comatroid. 


Now suppose that $M^c \cong P(U_{2,3},U_{2,3})$. Again, it is enough to show that $M/e$ is a binary comatroid for all $e$ in $E(M)$. If $e$ is in the projective closure of one of the $4$-circuits of $P(U_{2,3},U_{2,3})$, then $(M/e)^c$ has a coloop. Thus $(M/e)^c$  is disconnected with each component  having rank at most three, so it is a comatroid. If $e$ is not in one of these projective closures, then $(M/e)^c$ has a most one point and, again, $M/e$ is a comatroid.
\end{proof}

\section{Induced-restriction-minimal non-$GF(3)$-comatroids}
\label{tern}

In this section, we prove Theorem~\ref{main_ternary} and Corollary~\ref{cor_ternary}.

\begin{lemma}
\label{ternary_exceptional_case}
Let $M$ be a matroid in $\mathcal{M}_3$ such that $r(M) \geq 4$ and $M$ has a cocircuit of size less than four. Then $M$ can be obtained from a circuit of size at least three by $2$-summing a copy of $U_{2,4}$ to some, possibly empty, set of  elements of the circuit.
\end{lemma}

\begin{proof}
First we show the following. 

\begin{sublemma}
\label{helping_sublemma}
$M$ has no non-spanning circuit $C$ of rank at least three such that $C$ intersects a cocircuit of $M$ of size less than four.
\end{sublemma}

Assume that such a circuit $C$ exists and let $F$ be the projective flat spanned by $C$. Observe that $F$ has a cocircuit $A_F$ that has at most three green elements and so $F|R$ contains $A_F$ minus three points. Since $A_F$ is a ternary affine geometry of rank at least three, $F|R$ is connected of rank $r(F)$, a contradiction to the minimality of $M$. Thus \ref{helping_sublemma} holds. 

Now suppose that $M$ has a $2$-cocircuit $\{a,b\},$ say. Then $M = N \oplus_2 U_{2,3}$. If $N$ has an element $s$ parallel to the basepoint $p$ of the $2$-sum, then we move $s$ so that it is parallel to $p$ in $U_{2,3}$. 
Observe that if the basepoint $p$ is contained in a non-spanning circuit $D$ of $N$, then there is a non-spanning circuit $D'$ of $M$ that contains $\{a,b\}$ and has rank at least three, a contradiction to 
\ref{helping_sublemma}. 
Therefore $p$ is free in $N$. Thus, by Lemma~\ref{freedom}(ii), $N$ is obtained from a circuit of size at least three by $2$-summing a copy of $U_{2,4}$ to some of the elements of the circuit. If $\{a,b\}$ is in a triangle of $M$, then $M$ has a flat isomorphic to $N$. Thus $N$ is either a circuit of size at least four, or $N$ breaks up as a $2$-sum. By Corollary~\ref{cct}, or by Proposition~\ref{sum_of_connectivities} and Lemma~\ref{parallel_connection_complement}, $N$ is not a ternary comatroid, contradicting the minimality of $M$. 
Thus $\{a,b\}$ is not in a triangle of $M$, so $M$ satisfies the lemma. 

Next suppose that $M$ has a triad $\{a,b,c\},$ say. Observe that if $\{a,b,c\}$ is a triangle, we get our result as above. We may now assume that $\{a,b,c\}$ is independent. Let $\Pi$ be the projective plane spanned by $\{a,b,c\}$ and let $H$ be the projective hyperplane spanned by $E(M^c)-\{a,b,c\}$. It is clear that $\Pi|R$ is connected of rank three. Suppose that the projective lines spanned by $\{a,b\}, \{a,c\},$ and $\{b,c\}$ meet $H$ at $p, q,$ and $s$, respectively. Note that at most one of the points in $\{p,q,s\}$ is green otherwise $\Pi|G$ is connected of rank three, a contradiction. Thus we may assume that $q$ and $s$ are red.   We may also assume that $c$ is not free in $M$ otherwise we have the result by 
Lemma~\ref{freedom}(ii). Let $D$ be a non-spanning circuit of $M$ containing $c$. By orthogonality, $D$ contains either $a$ or $b$ and so $D$ has rank at least three. The result now follows by \ref{helping_sublemma}.
\end{proof}

\begin{lemma}
\label{rank_four_ternary}
Let $M$ be a matroid in $\mathcal{M}_3$ such that $r(M) = 4$. Then $M$ or $M^c$ has a cocircuit of size less than four.
\end{lemma}

\begin{proof}
Assume that neither $M$ nor $M^c$ has a cocircuit of size less than four. 

\begin{sublemma}
\label{ternary_hyperplane_bound_sublemma}
A rank-$3$ simple ternary matroid $N$ that is connected either has a connected rank-$3$ ternary complement or is $AG(2,3)\ba e$ or an extension of it. 
\end{sublemma}

Assume that $N^c$ is not connected or that $r(N^c) < 3$. Then $N^c$ is a restiction of $U_{2,4} \oplus U_{1,1}$. Thus $N$ is $AG(2,3)\ba e$ or an extension of it. Hence \ref{ternary_hyperplane_bound_sublemma} holds.

The next assertion  is an immediate consequence of Corollary~\ref{ternturn}.

\begin{sublemma}
\label{ternary_hyperplane_connected_sublemma}
Both $M$ and $M^c$ have a connected hyperplane.
\end{sublemma}

By \ref{ternary_hyperplane_bound_sublemma} and \ref{ternary_hyperplane_connected_sublemma}, it follows that we have both a red and a green triangle. In the arguments that follow, we will frequently exploit the  symmetry between $R$ and $G$. 

\begin{sublemma}
\label{5t+3_sublemma}
If a red triangle $T$ is contained in exactly $t$ red hyperplanes, then $|R| \geq 5t+3$.
\end{sublemma}

By \ref{ternary_hyperplane_bound_sublemma}, each red hyperplane containing $T$ has at least five red points not in the projective closure of $T$. The result is immediate.

\begin{sublemma}
\label{exactly_three_hyperplanes_sublemma}
Every red triangle $T$ is contained in exactly three red hyperplanes. Moreover, $|R| \geq 18$ and $|G| \geq 18$.
\end{sublemma}

Note that $T$ is  in at least two red hyperplanes otherwise we get a red cocircuit of size less than four. Assume that $T$ is in exactly two red hyperplanes, $H_1$ and $H_2$. Because each of $H_1$ and $H_2$ is $AG(2,3) \ba e$ or an extension of it, one can check that each element $t$ of $T$ is in a red triangle $T_i$ in $H_i$ that meets $T$ in $t$. Now consider the projective plane $\Pi$ that is spanned by $T_1$ and $T_2$. There are two green planes that contain $T$. Each of these has  $AG(2,3) \ba e$ as a restriction and meets $\Pi$ in a line through $t$. This line contains at least two green points. Hence $\Pi$ contains both a red $4$-circuit and a green $4$-circuit, a contradiction. We conclude that $T$ is in at least three red hyperplanes. Thus, by \ref{5t+3_sublemma}, $|R| \geq 18$. By symmetry, $|G| \geq 18$, so $|R| \leq 22$. Hence, by \ref{5t+3_sublemma} again, $T$ is not in four red hyperplanes. Therefore \ref{exactly_three_hyperplanes_sublemma} holds.

\begin{sublemma}
\label{black_4_point_line}
There is a red triangle that is not contained in a red $4$-point line.
\end{sublemma}

Suppose every red triangle is contained in a red $4$-point line. As every red hyperplane has   $AG(2,3) \ba e$ as a restriction, one easily checks that every red hyperplane is a $PG(2,3)$. Since every red triangle is in three red hyperplanes, it follows that $|R| \geq 31$, a contradiction to \ref{exactly_three_hyperplanes_sublemma}.

We now take a red triangle $T$ for which the fourth point, $g$, on the projective line spanned by $T$ is green. Now $T$ is in exactly three red  hyperplanes, $R_1, R_2,$ and $R_3,$ and one green hyperplane, $G_0$. Because $G_0$ contains at most one red point not in $T$, there are three lines, $G_1, G_2,$ and $G_3$, in $G_0$ that contain $g$ and at least two other green points.

We may assume that $|R| \leq |G|$, so $|R| \in \{18,19,20\}$. We may also assume that $|R_1 - T| \geq |R_2 - T| \geq |R_3 - T| \geq 5$. As $|R_1 - T| + |R_2 - T| + |R_3 - T| \in \{15,16,17\}$, we see that $|R_3 - T| = 5,$ that $|R_2 - T| \in \{5,6\},$ and that $|R_1 - T| \in \{5,6,7\}$. Thus $R_3$ contains a green triangle $T_1$ containing $g$, so each of the projective planes spanned by  $T_1 \cup G_1, T_1 \cup G_2,$ and $T_1 \cup G_3$ contains a green $4$-circuit. Moreover, each such plane meets each of $R_1$ and $R_2$ in a line through $g$. For each $i$ in $\{1,2\}$, the plane $R_i$ has at most one line through $g$ that does not contain at least two red points. Hence, for some $j$ in $\{1,2,3\}$, the projective plane spanned by $T_1 \cup G_j$ meets both $R_1$ and $R_2$ in a line through $g$ containing at least two red points. Then $T_1 \cup G_j$ contains a red $4$-circuit. As it contains a green $4$-circuit, we have a contradiction. 
\end{proof}

\begin{theorem}
\label{main_theorem_rank_bound2}
Let $M$ be a matroid in $\mathcal{M}_3$ such that $r(M) \geq 4$. Then $M$ or $M^c$ has a cocircuit of size less than four.
\end{theorem}

\begin{proof}
By Lemma \ref{rank_four_ternary}, the result holds when $r(M) = 4$. Therefore we may assume that $r(M) \geq 5$. Further assume that neither $M$ nor $M^c$ has a cocircuit of size less than four. We now let $P_{r(M)}$ denote the
ternary projective geometry of rank $r(M)$. A flat $F$ of $P_{r(M)}$ with $3 \leq r(F) < r(M)$ is \textbf{red} or \textbf{green} depending on whether $F|R$ or $F|G$ is connected of rank $r(F)$. Let $F$ be a rank-$(r-2)$ flat of $P_{r(M)}$. Then $F$ is contained in exactly four hyperplanes, say $H_1, H_2$, $H_3$, and $H_4$ of $P_{r(M)}$.
Moreover, as   neither $M$ nor $M^c$ has a cocircuit of size less than four, we deduce the following. 

\begin{sublemma}
\label{ternary_hsize2}
At least two of $H_1$, $H_2$, $H_3$, and $H_4$ have the same color as $F$.
\end{sublemma}

Now let $G_1$ and $R_1$ be the sets of green and red hyperplanes, respectively, of $P_{r(M)}$. The following is a straightforward consequence of Theorem \ref{oxwu_result}.

\begin{sublemma}
\label{ternary_nbr}
If $H \in G_1$, then at most one of the rank-$(r-2)$ projective flats contained in $H$ is red.
\end{sublemma}

Let $G_2$ and $R_2$ be the sets of green and red projective flats of $P_{r(M)}$ of rank $r-2$. We consider the bipartite graph $B$ with vertex classes $G_1 \cup R_1$ and $G_2 \cup R_2$ such that a vertex $X$ in $G_1 \cup R_1$ is incident to a vertex $Y$ in $G_2 \cup R_2$ if $Y \subseteq X$. As in the proof of Theorem \ref{main_theorem_rank_bound}, by \ref{ternary_nbr}, the number of cross edges of this graph is at most $|G_1| + |R_1|$. 

Each pair $(H_G, H_R)$, where $H_G \in G_1$ and $H_R \in R_1$, corresponds to a cross edge $e$. Note that at most three such pairs can correspond to this edge $e$. Thus the number of cross edges is at least $\frac{1}{3}|G_1||R_1|$, so $\frac{1}{3}|G_1||R_1| \leq |G_1|+|R_1|$. 
By symmetry, we may suppose that $|G_1| \geq |R_1|$. Thus $|R_1| \leq 3 + \frac{3|R_1|}{|G_1|} \leq 6$. Since $|G_1| + |R_1| = \frac{3^{r(M)}-1}{2}$ and $r(M) \geq 5$, we see that $|G_1| \geq 115$, so $|R_1|\leq 3 + \frac{3|R_1|}{|G_1|}$. Hence $|R_1|  \leq 3$ and $|G_1| \ge 118$. Since every red projective flat of rank $r(M)-2$ is contained in at least two red projective hyperplanes, it follows that $|R_2| \leq 3$. By \ref{ternary_hsize2}, a  flat in $R_2$ is contained in at most two hyperplanes in $G_1$ and so the number of $G_1R_2$-edges is at most six. Thus $\frac{1}{3}|G_1||R_1| \leq |R_1| + 6$, so $|R_1| \leq \frac{3|R_1|+18}{|G_1|}$. As $|R_1| \leq 3$ and $|G_1| \geq 118$, it follows that $|R_1|=0$, a contradiction to Lemma~\ref{ternturn}. 
\end{proof}

\begin{proof}[Proof of Theorem~\ref{main_ternary}.] 
A routine check shows that, up to complementation,   $U_{3,4},$ $P(U_{2,3},$ $U_{2,3}),$ $U_{2,4} \oplus_2 U_{2,3},$ $R_6,$ $P(U_{2,4},U_{2,3}),$ $M(K_4),$ and $\mathcal{W}^3$ are the only connected ternary matroids of rank three whose complements are also connected of rank three. 
Theorem~\ref{main_ternary}   now follows from Lemma~\ref{comat_small_rank}, Lemma~\ref{ternary_exceptional_case}, 
Lemma~\ref{rank_four_ternary}, and Theorem~\ref{main_theorem_rank_bound2}. 
\end{proof}

\begin{proof}[Proof of Corollary~\ref{cor_ternary}.] 
In view of Lemma~\ref{comat_small_rank} and Theorem~\ref{main_ternary}, it suffices to show that if $M^c$ is obtained from a circuit of size at least three by $2$-summing a copy of $U_{2,4}$ to some, possibly empty, set of elements of the circuit, then $M$ is an induced-minor-minimal ternary non-comatroid. But, when $M^c$ is as specified, $M/e$ has at most one point and so is a ternary comatroid.
\end{proof}

\section*{Acknowledgement}

The authors thank Guoli Ding for suggesting the study of comatroids.

\end{document}